\documentclass[a4paper,12pt]{amsart}
\usepackage{tikz,hyperref,pgf,amscd,enumerate,multicol}
\usepackage[all,cmtip]{xy}
\usepackage[margin=1.5cm]{geometry}

\usetikzlibrary{calc,3d,arrows}

\newtheorem{thm}{Theorem}[section]

\newtheorem{lem}[thm]{Lemma}
\newtheorem{prop}[thm]{Proposition}

\theoremstyle{definition}
\newtheorem{dfn}[thm]{Definition}
\newtheorem{dfns}[thm]{Definitions}

\theoremstyle{remark}
\newtheorem{rmk}[thm]{Remark}

\newtheorem{example}[thm]{Example}
\newtheorem{examples}[thm]{Examples}

\newcommand{\Depth}{2}
\newcommand{\Height}{2}
\newcommand{\Width}{2}

\newcommand{\Obj}{\operatorname{Obj}}
\newcommand{\Sk}{\operatorname{Sk}}

\newcommand{\Gg}{\mathcal{G}}
\newcommand{\Cc}{\mathcal{C}}
\newcommand{\Pp}{\mathcal{P}}

\newcommand{\ZZ}{\mathbb{Z}}
\newcommand{\NN}{\mathbb{N}}
\newcommand{\FF}{\mathbb{F}}

\def\IoIIdimdots(#1/#2/#3,#4){\node at (#1,#4) {$.$};\node at (#2,#4) {$.$};\node at (#3,#4) {$.$};}
\def\IIoIIdimdots(#1,#2/#3/#4){\node at (#1,#2) {$.$};\node at (#1,#3) {$.$};\node at (#1,#4) {$.$};}

\def\IoIIIdimdots(#1/#2/#3,#4,#5){\node at (#1,#4,#5) {$.$};\node at (#2,#4,#5) {$.$};\node at (#3,#4,#5) {$.$};}
\def\IIoIIIdimdots(#1,#2/#3/#4,#5){\node at (#1,#2,#5) {$.$};\node at (#1,#3,#5) {$.$};\node at (#1,#4,#5) {$.$};}
\def\IIIoIIIdimdots(#1,#2,#3/#4/#5){\node at (#1,#2,#3) {$.$};\node at (#1,#2,#4) {$.$};\node at (#1,#2,#5) {$.$};}

\title{Computing the fundamental group of a higher-rank graph}
\author{Sooran Kang}
\address{Sooran Kang, College of General Education, Chung-Ang University, Seoul 06974, Republic of Korea }
\email{sooran09@cau.ac.kr}

 \author{David Pask}
 \address{David Pask \\ School of Mathematics and
Applied Statistics  \\
The University of Wollongong\\
NSW  2522\\
AUSTRALIA} \email{dpask@uow.edu.au}
 
 \author{Samuel B.G. Webster}
 \address{Samuel B.G. Webster, Independent Hospital Pricing Authority, Level 6, 1 Oxford Street, Sydney, NSW 2000, Australia}
 \email{sbgwebster@gmail.com}

\thanks{This research was supported by the National Research Foundation of Korea (NRF) grant funded by the Korea government (MSIT), No. NRF-2020R1F1A101076072 and the Australian Reseach Council.} 

\keywords{Higher-rank graph,  fundamental group}
\subjclass[2010]{Primary: {14F35}; Secondary: {18D99}}

\date{\today}

\begin{document}

\begin{abstract}
We compute a presentation of the fundamental group of a higher-rank graph using a coloured graph description of higher-rank graphs developed by the third author. We compute the fundamental groups of several examples from the literature. Our results fit naturally into the suite of known geometrical results about $k$-graphs when we show that the abelianisation of fundamental group is the homology group. We end with a calculation which gives a non-standard presentation of the fundamental group of the Klein bottle to the one normally found in the literature.
\end{abstract}

\maketitle

\section{Introduction}

\noindent
Higher-rank graphs, or $k$-graphs, were introduced in \cite{KP1} as a graphical approach to the higher-rank Cuntz-Krieger algebras introduced by Robertson and Steger in \cite{RS}. Since then $k$-graphs have been studied by many authors from several points of view (see \cite{A-PCaHR, CKSS, CFaH,  E, FGJKP2, aHLRS, RSY1, Y1}, for example). The motivation for this paper comes from the recent developments of \cite{KKQS,KPS3,KPS4,KPS5} and \cite{PQR1} where the geometric nature of $k$-graphs has been investigated and then used to construct new families of twisted $k$-graph $C^*$-algebras. In \cite{KKQS} it was shown that a $k$-graph may be realised as a topological space in a way which preserves homotopy type. It is therefore of interest to provide a facility to computate the fundamental group of a $k$-graph. This is our main purpose. 

For $1$-graphs, or directed graphs, the result is classical and easy to compute, as the universal covering object is a tree (see \cite{St}). For $k$-graphs the situation is more complicated, the fundamental group has relations arising from the factorisation property and so they are not free groups. Key to our analysis is the use of coloured graphs to describe the structure of a $k$-graph (see \cite{HRSW}). The $1$-skeleton $\Sk_\Lambda$ of a $k$-graph $\Lambda$ is a $k$-coloured graph which, together with some additional combinatorial data $\mathcal{C}$, which
encodes the natural quotient structure of $\Lambda$.

 After introducing our conventions, we describe the relationship between $k$-graphs and $k$-coloured graphs in sections~\ref{sec:prelim1} and~\ref{sec:prelim2}. In section~\ref{sec:pi1} we give a presentation of the fundamental group of a $k$-graph in terms of the fundamental group of its $1$-skeleton (see Theorem~\ref{thm:fgc}). To illustrate the efficacy of our result we give several computations in Examples~\ref{ex:kps_ex}. Finally, in section~\ref{sec:h1} we show that the abelianisation of the fundamental group of a $k$-graph agrees with its homology as defined in \cite{KPS3}. Then, in Example~\ref{ex:isklein}, we compute the fundamental group of a $2$-graph from the Klein bottle example in \cite[Example 3.13]{KKQS} which reveals a non standard presentation of this group.

\section{Conventions}

For $k \ge 1$ let $\mathbb{N}^k$ denote the monoid of $k$-tuples of natural numbers under addition, and denote the canonical generators by $e_1,e_2 , \ldots, e_k$. 
For $m \in \NN^k$ we write $m = \sum_{i=1}^k m_i e_i$ then for $m,n \in \NN^k$ we say that $m \le n$ if and only if $m_i \le n_i$ for $i=1, \ldots , k$.

A \emph{directed graph} $E$ is a quadruple $(E^0,E^1,r_E,s_E)$, where $E^0$ is the set of \emph{vertices}, $E^1$ is the set of \emph{edges}, and $r_E,s_E:E^1 \to E^0$ are \emph{range} and \emph{source} maps, giving a direction to each edge (if there is no chance of confusion we will drop the subscripts). We follow the conventions of~\cite{Raeburn2005} which are suited to the categorical setting we wish to pursue: a \emph{path of length n} is a sequence $\mu=\mu_1\mu_2\cdots\mu_n$ of edges such that $s(\mu_i)=r(\mu_{i+1})$ for $1\le i\le n-1$. We denote by $E^n$ the set of all paths of length $n$, and define $E^*=\bigcup_{n\in\mathbb{N}}E^n$. We extend $r$ and $s$ to $E^*$ by setting $r(\mu)=r(\mu_1)$ and $s(\mu)=s(\mu_n)$.

To align with the established literature on the fundamental groupoid of a $k$-graph, the following definitions are taken from \cite[Definition 5.1]{PQR1}: Let $E$ be a directed graph. For each $e \in E^1$ we introduce an \emph{inverse} edge $e^{-1}$ with
$s ( e^{-1} ) = r ( e )$ and $r ( e^{-1} ) = s ( e)$.
 Let $E^{-1} = \{ e^{-1} : e \in E^1 \}$ and $E^u=E^1\cup E^{-1}$, then $E^+ = ( E^0 , E^u , r , s)$ is a directed graph called the \textit{augmented graph of $E$}.  Let $\mathcal{P}(E^+)$ be the path category of $E^+$. If we set $(e^{-1})^{-1}=e$ then $e^{-1} \in \mathcal{P} (E^+)$ for any $e \in E^u$ and by extension $\lambda^{-1} = \lambda_n^{-1} \cdots \lambda_1^{-1} \in \mathcal{P} ( E^+ )$ for any $\lambda = \lambda_1 \cdots \lambda_n \in \mathcal{P} ( E^+ )$. Elements of $\mathcal{P} ( E^+ )$ are called \textit{undirected paths in $E$}, and elements of $\mathcal{P} ( E^+ )$ which do not contain $e e^{-1}$ for any $e \in E^u$ are called \emph{reduced} undirected paths in $E$ (vertices are reduced paths). 
 
 Let $E$ be a directed graph then $E$ is \textit{connected} if for every $u,v \in E^0$ there is $\alpha \in \mathcal{P} ( E^+ )$ with $u = s ( \alpha )$ and $v = r ( \alpha )$. A tree $T$ is a connected directed graph such that the only reduced $\alpha \in \mathcal{P} ( T^+ )$ with the same source and range are vertices. Let $E$ be a directed graph with subgraph $T$ which is a tree, then $T$ is a \textit{maximal spanning tree} if $T^0 = E^0$. Every connected directed graph has a (not necessarily unique) maximal spanning tree (see \cite[\S 2.1.5]{St}).
 
 \begin{rmk}\label{rmk: fundamental group ntn} Fix a maximal spanning tree $T$ of a connected directed graph $E$ and $v \in E^0$. For each $w \in E^0$ there is a unique reduced path $\eta_w$ in the augmented graph $T^+$ from $v$ to $w$, which is an element of the fundamental groupoid $\mathcal{G} ( E )$. For $\lambda \in \Pp ( E^+ )$ define $\zeta_\lambda = \eta_{r(\lambda)}^{-1} \lambda \eta_{s(\lambda)} \in v \mathcal{P} (E^+) v$. Note that $\zeta_\lambda^{-1} = \zeta_{\lambda^{-1}}$. 
\end{rmk}

\section{Coloured graphs} \label{sec:prelim1}
For $k \ge 1$, a \emph{$k$-coloured graph} $E$ is a directed graph along with a \emph{colour map} $c_E:E^1 \to \{c_1,\dots,c_k\}$. By considering $\{c_1,\dots , c_k\}$ as generators of the free group  $\mathbb{F}_k$, we may extend $c_E : E^* \setminus E^0 \to \mathbb{F}_k^+$ by $c_E (\mu_1 \cdots \mu_n)=c_E (\mu_1)c_E (\mu_2) \cdots c_E (\mu_n)$. We will drop the subscript from $c_E$ if there is no risk of confusion. For $k$-coloured directed graphs $E$ and $F$, a \emph{coloured-graph morphism} $\phi: F \to E$ is a graph morphism satisfying $c_E \circ \phi^1 = c_F$.

For $2$-coloured graphs, the convention is to draw edges with colour $c_1$ blue (or solid) and edges with colour $c_2$ red (or dashed).

\noindent\begin{minipage}[c]{0.78\linewidth}
\begin{example}\label{ex:model coloured graphs} 
For $k \ge 1$ and $m\in \mathbb{N}^k$, the $k$-coloured graph $E_{k,m}$ is defined by 
$E_{k,m}^0=\{n\in\mathbb{N}^k:0\le n\le m\}$, $E_{k,m}^1=\{\varepsilon_i^n : n,n+e_i\in E_{k,m}^0\}$, with $r(\varepsilon_i^n)=n$, $s(\varepsilon_i^n)=n+e_i$ and $c_{E}(\varepsilon_i^n)=c_i$. 
The 2-coloured graph $E_{2,e_1+e_2}$ is used often, and is depicted to the right. 
\end{example}
\end{minipage}
\begin{minipage}[c]{0.2\linewidth}
\[
\begin{tikzpicture}[>=stealth,yscale=1.5,xscale=2]
\node (00) at (0,0) {$0$};
\node (10) at (1,0) {$e_1$}
	edge[->,blue] node[auto,black] {$\varepsilon_1^{0}$} (00);
\node (01) at (0,1) {$e_2$}
	edge[->,red,dashed] node[auto,black,swap] {$\varepsilon_2^{0}$} (00);
\node (11) at (1,1) {$e_1+e_2$}
	edge[->,blue] node[auto,black,swap] {$\varepsilon_1^{e_2}$} (01)
	edge[->,red,dashed] node[auto,black] {$\varepsilon_2^{e_1}$} (10);
\end{tikzpicture}
\]
\end{minipage}

\noindent
Let $E$ be a $k$-coloured graph and $i \neq j \leq k$. A coloured graph morphism $\phi: E_{k,e_i+e_j} \to E$ is called an \emph{square} in $E$. One may represent a square $\phi$ as a labelled version of $E_{k,e_i+e_j}$. For instance the $2$-coloured graph below on the left has only one square $\phi$, shown to its right, given by $\phi(n) = v$ for all $n \in E_{2,{e_1+e_2}}^0$, $\phi( \varepsilon_1^0) = \phi( \varepsilon_1^{e_2}) = e$ and $\phi( \varepsilon_2^0 ) = \phi( \varepsilon_2^{e_1} ) = f$. 
\[
\begin{tikzpicture}[>=stealth,scale=1.2]
\node[inner sep=0.8pt] (v) at (0,0) {$v$}
	edge[loop,blue,->,in=45,out=-45,looseness=15] node[auto,black,swap] {$e$} (v)
	edge[loop,red,dashed,->,in=135,out=-135,looseness=15] node[auto,black] {$f$} (v);
\begin{scope}[xshift=2.5cm,yshift=-0.5cm]
\node at (0.5,0.5) {$\phi$};
\node[inner sep=0.8pt] (00) at (0,0) {$v$};
\node[inner sep=0.8pt] (10) at (1,0) {$v$}
	edge[->,blue] node[auto,black] {$e$} (00);
\node[inner sep=0.8pt] (01) at (0,1) {$v$}
	edge[->,red,dashed] node[auto,black,swap] {$f$} (00);
\node[inner sep=0.8pt] (11) at (1,1) {$v$}
	edge[->,blue] node[auto,black,swap] {$e$} (01)
	edge[->,red,dashed] node[auto,black] {$f$} (10);
\end{scope}
\end{tikzpicture}
\]

\noindent
$\Cc_E = \{ \phi : E_{k,e_i+e_j} \to E : 1 \le i \neq j \le k  \}$ denotes the set of squares in a $k$-coloured graph $E$.

A collection of squares $\Cc$ in a $k$-coloured graph $E$ is called \emph{complete} if for each $i \neq j \leq k$ and $c_ic_j$-coloured path $fg \in E^2$, there exists a unique $\phi \in \Cc$ such that $\phi(\varepsilon_i^{0}) = f$ and $\phi(\varepsilon_j^{e_i}) = g$. In this case, uniqueness of $\phi$ gives a unique $c_jc_i$-coloured path $g'f'$ with $g' = \phi(\varepsilon_j^{0})$ and $f' = \phi(\varepsilon_i^{e_j})$. We will write $fg \sim_\Cc g'f'$ and refer to elements $(fg,g'f')$ of this relation as \emph{commuting squares}. 

\begin{example} \label{kgex2} 
For $n \ge 1$ define $\underline{n} = \{ 1 , \ldots , n \}$. For $m,n \ge 1$, let  $\theta : \underline{m} \times \underline{n} \to \underline{m} \times \underline{n}$ be a bijection. Let $E_\theta$ be the $2$-coloured graph with $E_\theta ^0= \{v\}$, $E_\theta^1 = \{f_1 , \ldots , f_m , g_1 , \ldots , g_n\}$, and colouring map $c : E^1_{\theta} \to \{c_1,c_2\}$ by $c ( f_i ) = c_1$ for $i \in \underline{m} $ and $c (g_j ) = c_2$ for $j \in \underline{n}$. For each $(i,j) \in \underline{m} \times \underline{n}$, define $\phi_{(i,j)} : E_{2,e_1+e_2} \to E_\theta$ by
\[
\phi_{(i,j)} ( \varepsilon_1^{0} ) = f_i, \ \ \phi_{(i,j)} (\varepsilon_1^{e_2} ) = f_{i'},\ \
\phi_{(i,j)} ( \varepsilon_2^{0} )= g_{j'},  \text{ and } \phi_{(i,j)} ( \varepsilon_2^{e_1} ) = g_j , \text{ where $\theta ( i , j ) = ( i' , j' ).$}
\]

\noindent As $\theta$ is a bijection $\mathcal{C}_{E_\theta} = \{\phi_{(i,j)} : (i,j) \in \underline{m} \times \underline{n}\}$
is a complete collection of squares. 
\end{example}

\section{Higher-rank graphs} \label{sec:prelim2}

A \emph{$k$-graph} (or a \emph{higher-rank graph}) is a countable category $\Lambda$ with a \emph{degree} functor $d:\Lambda \to \NN^k$ satisfying the \emph{factorisation property}: if $\lambda\in\Lambda$ and $m,n\in\NN^k$ are such that $d(\lambda)=m+n$, then there are unique $\mu,\nu \in \Lambda$ with $d(\mu)=m$, $d(\nu)=n$ and $\lambda=\mu\nu$. 

Given $m \in \NN^k$ we define $\Lambda^m := d^{-1}(m)$. Given $v,w \in \Lambda^0$ and $F \subseteq \Lambda$ define $vF := r^{-1}(v) \cap F$, $Fw := s^{-1}(w) \cap F$, and $vFw := vF \cap Fw$. The factorisation property allows us to identify $\Lambda^0$ with $\Obj(\Lambda)$, and we call its elements \emph{vertices}.

By the factorisation property, for each $\lambda\in\Lambda$ and $m \leq n \leq d(\lambda)$, we may write $\lambda=\lambda'\lambda''\lambda'''$, where $d(\lambda')=m, d(\lambda'') = n-m$ and $d(\lambda'')=d(\lambda)-n$; then $\lambda (m,n) :=\lambda''$. For more information about $k$-graphs see \cite{HRSW,KP1,RSY1} for example.

\begin{examples} \label{ex:kg}
\begin{enumerate}[(a)]
\item Let $E$ be a directed graph. The collection $E^*$ of finite paths in $E$ forms a category, called the \emph{path category} of $E$, denoted by $\mathcal{P}(E)$. The map $d : \mathcal{P} (E) \to \NN$ defined by $d(\mu)=n$ if and only if $\mu \in E^n$ is a functor which satisfies the factorisation property, hence $( \mathcal{P}(E),d)$ is a $1$-graph. It turns out that every $1$-graph arises in this way (see \cite[Example 1.3]{KP1}). 
\item For $k \ge 1$; let $\Delta_k = \{ (m,n) : m,n \in \mathbb{Z}^k : m \le n \}$. With structure maps $r(m,n)=m , s (m,n)=n$, so
that $( \ell,n) = ( \ell,m) (m,n)$ then $\Delta_k$ is a category. Set $d ( m , n ) = n - m$, then $d$ is
a functor and $( \Delta_k , d )$ is a $k$-graph. 
The vertices $\Delta_k^0 = \{ (m,m) : m \in \mathbb{Z}^k  \}$
may be identified with $\mathbb{Z}^k$.
\item Resuming the notation of Example~\ref{kgex2} let $\theta : \underline{m} \times \underline{n} \to \underline{m} \times \underline{n}$ be a bijection. Let $\FF^2_\theta$ be the semigroup with generators $\{f_1 , \ldots , f_m , g_1 , \ldots g_n \}$ and relations $f_i g_j = g_{j'} f_{i'}$ where $\theta (i, j) = (i' , j' )$ for $(i, j) \in \underline{m} \times \underline{n}$. Let $d ( f_i ) = e_1$ for $i=1,\ldots,m$ and $d ( g_j ) = e_2$ for $j=1,\ldots ,n$ then $d$ extends to a functor from $\FF^2_\theta$ to $\NN^2$ with the factorisation property, and so $\FF^2_\theta$ is a $2$-graph (see \cite[\S 2]{Y1}).

\item Recall from \cite{FPS2009} that if $\Lambda$ is a $k$-graph and $\alpha$ is an automorphism of $\Lambda$,
then there is a $(k + 1)$-graph $\Lambda \times_\alpha  \ZZ$ with morphisms $\Lambda \times \NN$, range and source maps given
by $r(\lambda, n) = (r(\lambda), 0)$, $s(\lambda, n) = (\alpha^{-n} (s(\lambda)), 0)$, degree map given by $d(\lambda, n) = (d(\lambda), n)$ and
composition given by $(\lambda, m)(\mu, n) := (\lambda \alpha^{m} (\mu), m + n)$. In particular $(\Lambda \times_\alpha \ZZ)^0 = \Lambda^0 \times \{0\}$.
\end{enumerate}
\end{examples}

\noindent
We define the \emph{$1$-skeleton} of a $k$-graph $\Lambda$ to be the $k$-coloured graph $\Sk_\Lambda$ given by $\Sk_\Lambda^0 = \Obj(\Lambda)$, $\Sk_\Lambda^1 = \bigcup_{i\leq k}\Lambda^{e_i}$, with range and source as in $\Lambda$. The colouring map $c : \Sk_\Lambda^1 \to \{c_1,\dots,c_k\}$ is given by $c(f) = c_i$ if and only if $f \in \Lambda^{e_i}$. The $1$-skeleton $\Sk_\Lambda$ comes with a canonical set of squares $\Cc_\Lambda := \{\phi_\lambda : \lambda \in \Lambda^{e_i+e_j}: i \neq j \leq k  \}$, where $\phi_\lambda : E_{k,e_i+e_j} \to \Sk_\Lambda$ is given by $\phi_\lambda(\varepsilon_\ell^{n}) = \lambda(n,n+e_\ell)$ for each $n \leq e_i+e_j$ and $\ell = i,j$. The collection $\Cc_\Lambda$ is complete by \cite[Lemma 4.2]{HRSW}.

Conversely, in \cite[Theorem 4.4,Theorem 4.5]{HRSW} it is shown that for a $k$-coloured graph $E$ with a complete, associative\footnote{The associative condition which only applies if $k \ge 3$ is quite complicated, and we will not deal with it here. For more details see \cite[\S 3]{HRSW}} collection of squares $\Cc_E$ determines a unique $k$-graph $\Lambda_{E,\Cc_E}$.

\begin{examples} \label{ex:morekg}
\begin{enumerate}[(a)]
\item The $2$-graph $\Lambda_{E_\theta , \Cc_{E_\theta}}$, determined by the 2-coloured graph $E_\theta$ with  squares  $\Cc_{E_\theta}$ described in Example~\ref{kgex2} is isomorphic to $\FF^2_\theta$ defined in Examples~\ref{ex:kg} (c) .

\item Recall the $k$-graph $\Delta_k$ described in Examples~\ref{ex:kg}. Part of the $1$-skeleton of $\Delta_3$, as seen from the first octant is shown below:
\[
\begin{tikzpicture}[->=stealth,scale=1.3]
        \node[inner sep=0.25pt, circle, opacity=0] (vert000) at (0,0,0) [draw] {.};
        \node[inner sep=0.25pt, circle, opacity=0] (vert100) at (1,0,0) [draw] {.};
        \node[inner sep=0.25pt, circle, opacity=0] (vert200) at (2,0,0) [draw] {.};
        \node[inner sep=0.25pt, circle, opacity=0] (vert300) at (3,0,0) [draw] {.};
        \node[inner sep=0.25pt, circle, opacity=0] (vert010) at (0,1,0) [draw] {.};
        \node[inner sep=0.25pt, circle, opacity=0] (vert110) at (1,1,0) [draw] {.};
        \node[inner sep=0.25pt, circle, opacity=0] (vert210) at (2,1,0) [draw] {.};
        \node[inner sep=0.25pt, circle, opacity=0] (vert310) at (3,1,0) [draw] {.};
        \node[inner sep=0.25pt, circle, opacity=0] (vert020) at (0,2,0) [draw] {.};
        \node[inner sep=0.25pt, circle, opacity=0] (vert120) at (1,2,0) [draw] {.};
        \node[inner sep=0.25pt, circle, opacity=0] (vert220) at (2,2,0) [draw] {.};
        \node[inner sep=0.25pt, circle, opacity=0] (vert320) at (3,2,0) [draw] {.};
        \node[inner sep=0.25pt, circle, opacity=0] (vert001) at (0,0,1) [draw] {.};
        \node[inner sep=0.25pt, circle, opacity=0] (vert101) at (1,0,1) [draw] {.};
        \node[inner sep=0.25pt, circle, opacity=0] (vert201) at (2,0,1) [draw] {.};
        \node[inner sep=0.25pt, circle, opacity=0] (vert301) at (3,0,1) [draw] {.};
        \node[inner sep=0.25pt, circle, opacity=0] (vert011) at (0,1,1) [draw] {.};
        \node[inner sep=0.25pt, circle, opacity=0] (vert111) at (1,1,1) [draw] {.};
        \node[inner sep=0.25pt, circle, opacity=0] (vert211) at (2,1,1) [draw] {.};
        \node[inner sep=0.25pt, circle, opacity=0] (vert311) at (3,1,1) [draw] {.};
        \node[inner sep=0.25pt, circle, opacity=0] (vert021) at (0,2,1) [draw] {.};
        \node[inner sep=0.25pt, circle, opacity=0] (vert121) at (1,2,1) [draw] {.};
        \node[inner sep=0.25pt, circle, opacity=0] (vert221) at (2,2,1) [draw] {.};
        \node[inner sep=0.25pt, circle, opacity=0] (vert321) at (3,2,1) [draw] {.};
      
            \node[inner sep=0pt, circle, fill=black] at (vert000) [draw] {.};
            \node[inner sep=0.25pt, anchor = north west] at (vert000.south east) {\tiny $\scriptscriptstyle (0,0,0)$};

            \node[inner sep=0pt, circle, fill=black] at (vert001) [draw] {.};
            \node[inner sep=0.25pt, anchor = north west] at (vert001.south east) {\tiny $\scriptscriptstyle (1,0,0)$};

            \node[inner sep=0pt, circle, fill=black] at (vert010) [draw] {.};
            \node[inner sep=0.25pt, anchor = north west] at (vert010.south east) {\tiny $\scriptscriptstyle (0,0,1)$};

            \node[inner sep=0pt, circle, fill=black] at (vert011)  [draw] {.};
            \node[inner sep=0.25pt, anchor = north west] at (vert011.south east) {\tiny $\scriptscriptstyle (1,0,1)$};

            \node[inner sep=0pt, circle, fill=black] at (vert020) [draw] {.};
            \node[inner sep=0.25pt, anchor = north west] at (vert020.south east) {\tiny $\scriptscriptstyle (0,0,2)$};

            \node[inner sep=0pt, circle, fill=black] at (vert021) [draw] {.};
            \node[inner sep=0.25pt, anchor = north west] at (vert021.south east) {\tiny $\scriptscriptstyle (1,0,2)$};

            \node[inner sep=0pt, circle, fill=black] at (vert100) [draw] {.};
            \node[inner sep=0.25pt, anchor = north west] at (vert100.south east) {\tiny $\scriptscriptstyle (0,0,1)$};

            \node[inner sep=0pt, circle, fill=black] at (vert101) [draw] {.};
            \node[inner sep=0.25pt, anchor = north west] at (vert101.south east) {\tiny $\scriptscriptstyle (1,1,0)$};

            \node[inner sep=0pt, circle, fill=black] at (vert110) [draw] {.};
            \node[inner sep=0.25pt, anchor = north west] at (vert110.south east) {\tiny $\scriptscriptstyle (0,1,1)$};

            \node[inner sep=0pt, circle, fill=black] at (vert111) [draw] {.};
            \node[inner sep=0.25pt, anchor = north west] at (vert111.south east) {\tiny $\scriptscriptstyle (1,1,1)$};

            \node[inner sep=0pt, circle, fill=black] at (vert120) [draw] {.};
            \node[inner sep=0.25pt, anchor = north west] at (vert120.south east) {\tiny $\scriptscriptstyle (0,1,2)$};

            \node[inner sep=0pt, circle, fill=black] at (vert121) [draw] {.};
            \node[inner sep=0.25pt, anchor = north west] at (vert121.south east) {\tiny $\scriptscriptstyle (1,1,2)$};

            \node[inner sep=0pt, circle, fill=black] at (vert200) [draw] {.};
            \node[inner sep=0.25pt, anchor = north west] at (vert200.south east) {\tiny $\scriptscriptstyle (0,0,2)$};

            \node[inner sep=0pt, circle, fill=black] at (vert201) [draw] {.};
            \node[inner sep=0.25pt, anchor = north west] at (vert201.south east) {\tiny $\scriptscriptstyle (1,2,0)$};

            \node[inner sep=0pt, circle, fill=black] at (vert210) [draw] {.};
            \node[inner sep=0.25pt, anchor = north west] at (vert210.south east) {\tiny $\scriptscriptstyle (0,2,1)$};

            \node[inner sep=0pt, circle, fill=black] at (vert211) [draw] {.};
            \node[inner sep=0.25pt, anchor = north west] at (vert211.south east) {\tiny $\scriptscriptstyle (1,2,1)$};

            \node[inner sep=0pt, circle, fill=black] at (vert220) [draw] {.};
            \node[inner sep=0.25pt, anchor = north west] at (vert220.south east) {\tiny $\scriptscriptstyle (0,2,2)$};

            \node[inner sep=0pt, circle, fill=black] at (vert221) [draw] {.};
            \node[inner sep=0.25pt, anchor = north west] at (vert221.south east) {\tiny $\scriptscriptstyle (1,2,2)$};

            \node[inner sep=0pt, circle, fill=black] at (vert300) [draw] {.};
            \node[inner sep=0.25pt, anchor = north west] at (vert300.south east) {\tiny $\scriptscriptstyle (0,3,0)$};

            \node[inner sep=0pt, circle, fill=black] at (vert301) [draw] {.};
            \node[inner sep=0.25pt, anchor = north west] at (vert301.south east) {\tiny $\scriptscriptstyle (1,3,0)$};

            \node[inner sep=0pt, circle, fill=black] at (vert310) [draw] {.};
            \node[inner sep=0.25pt, anchor = north west] at (vert310.south east) {\tiny $\scriptscriptstyle (0,3,1)$};

            \node[inner sep=0pt, circle, fill=black] at (vert311) [draw] {.};
            \node[inner sep=0.25pt, anchor = north west] at (vert311.south east) {\tiny $\scriptscriptstyle (1,3,1)$};

            \node[inner sep=0pt, circle, fill=black] at (vert320) [draw] {.};
            \node[inner sep=0.25pt, anchor = north west] at (vert320.south east) {\tiny $\scriptscriptstyle (0,3,2)$};

            \node[inner sep=0pt, circle, fill=black] at (vert321) [draw] {.};
            \node[inner sep=0.25pt, anchor = north west] at (vert321.south east) {\tiny $\scriptscriptstyle (1,3,2)$};
        
        
            \IoIIIdimdots(3.3/3.45/3.6,0,0)
            \IoIIIdimdots(3.3/3.45/3.6,1,0)
            \IoIIIdimdots(3.3/3.45/3.6,2,0)
            \IoIIIdimdots(3.3/3.45/3.6,0,1)
            \IoIIIdimdots(3.3/3.45/3.6,1,1)
            \IoIIIdimdots(3.3/3.45/3.6,2,1)
            \IIoIIIdimdots(0,2.3/2.45/2.6,0)
            \IIoIIIdimdots(1,2.3/2.45/2.6,0)
            \IIoIIIdimdots(2,2.3/2.45/2.6,0)
            \IIoIIIdimdots(3,2.3/2.45/2.6,0)
            \IIoIIIdimdots(0,2.3/2.45/2.6,1)
            \IIoIIIdimdots(1,2.3/2.45/2.6,1)
            \IIoIIIdimdots(2,2.3/2.45/2.6,1)
            \IIoIIIdimdots(3,2.3/2.45/2.6,1)
            \IIIoIIIdimdots(0,0,-0.5/-0.8/-1.1)
            \IIIoIIIdimdots(1,0,-0.5/-0.8/-1.1)
            \IIIoIIIdimdots(2,0,-0.5/-0.8/-1.1)
            \IIIoIIIdimdots(3,0,-0.5/-0.8/-1.1)
            \IIIoIIIdimdots(0,1,-0.5/-0.8/-1.1)
            \IIIoIIIdimdots(1,1,-0.5/-0.8/-1.1)
            \IIIoIIIdimdots(2,1,-0.5/-0.8/-1.1)
            \IIIoIIIdimdots(3,1,-0.5/-0.8/-1.1)
            \IIIoIIIdimdots(0,2,-0.5/-0.8/-1.1)
            \IIIoIIIdimdots(1,2,-0.5/-0.8/-1.1)
            \IIIoIIIdimdots(2,2,-0.5/-0.8/-1.1)
            \IIIoIIIdimdots(3,2,-0.5/-0.8/-1.1)
        
        \draw[style=semithick, red] (vert100.west)--(vert000.east);
        \draw[style=semithick, red] (vert110.west)--(vert010.east);
        \draw[style=semithick, red] (vert120.west)--(vert020.east);
        \draw[style=semithick, red] (vert101.west)--(vert001.east);
        \draw[style=semithick, red] (vert200.west)--(vert100.east);
        \draw[style=semithick, red] (vert210.west)--(vert110.east);
        \draw[style=semithick, red] (vert220.west)--(vert120.east);
        \draw[style=semithick, red] (vert201.west)--(vert101.east);
        \draw[style=semithick, red] (vert300.west)--(vert200.east);
        \draw[style=semithick, red] (vert310.west)--(vert210.east);
        \draw[style=semithick, red] (vert320.west)--(vert220.east);
        \draw[style=semithick, red] (vert301.west)--(vert201.east);
%
        \draw[style=semithick, green!50!black] (vert010.south)--(vert000.north);
      \draw[style=semithick, green!50!black] (vert011.south)--(vert001.north);
        \draw[style=semithick, green!50!black] (vert110.south)--(vert100.north);
        \draw[style=semithick, green!50!black] (vert111.south)--(vert101.north);
       \draw[style=semithick, green!50!black] (vert210.south)--(vert200.north);
        \draw[style=semithick, green!50!black] (vert211.south)--(vert201.north);
        \draw[style=semithick, green!50!black] (vert310.south)--(vert300.north);
        \draw[style=semithick, green!50!black] (vert311.south)--(vert301.north);
       \draw[style=semithick, green!50!black] (vert020.south)--(vert010.north);
        \draw[style=semithick, green!50!black] (vert021.south)--(vert011.north);
        \draw[style=semithick, green!50!black] (vert120.south)--(vert110.north);
        \draw[style=semithick, green!50!black] (vert121.south)--(vert111.north);
        \draw[style=semithick, green!50!black] (vert220.south)--(vert210.north);
        \draw[style=semithick, green!50!black] (vert221.south)--(vert211.north);
        \draw[style=semithick, green!50!black] (vert320.south)--(vert310.north);
        \draw[style=semithick, green!50!black] (vert321.south)--(vert311.north);
        \draw[style=semithick, red] (vert111.west)--(vert011.east);
        \draw[style=semithick, red] (vert121.west)--(vert021.east);
        \draw[style=semithick, red] (vert211.west)--(vert111.east);
        \draw[style=semithick, red] (vert221.west)--(vert121.east);
        \draw[style=semithick, red] (vert311.west)--(vert211.east);
        \draw[style=semithick, red] (vert321.west)--(vert221.east);
%
        \draw[style=semithick, blue] (vert001.north east) -- (vert000.south west);
        \draw[style=semithick, blue] (vert011.north east) -- (vert010.south west);
        \draw[style=semithick, blue] (vert021.north east) -- (vert020.south west);
        \draw[style=semithick, blue] (vert101.north east) -- (vert100.south west);
        \draw[style=semithick, blue] (vert111.north east) -- (vert110.south west);
        \draw[style=semithick, blue] (vert121.north east) -- (vert120.south west);
        \draw[style=semithick, blue] (vert201.north east) -- (vert200.south west);
        \draw[style=semithick, blue] (vert211.north east) -- (vert210.south west);
        \draw[style=semithick, blue] (vert221.north east) -- (vert220.south west);
        \draw[style=semithick, blue] (vert301.north east) -- (vert300.south west);
        \draw[style=semithick, blue] (vert311.north east) -- (vert310.south west);
        \draw[style=semithick, blue] (vert321.north east) -- (vert320.south west);
    \end{tikzpicture}
\]

\noindent
It is straightforward to see that $\Sk_{\Delta_3}^0 = \mathbb{Z}^3$, $\Sk_{\Delta_3}^1 = \{ ( m , m+e_j ) : m \in \mathbb{Z}^3 , 1 \le j \le 3 \}$, $r(m,m+e_j)=m$ and $s (m,m+e_j)=m+e_j$. The commuting squares are 
\[
\Cc = \{ (m,m+e_i) (m+e_i,m+e_i+e_j) = (m,m+e_j)(m+e_i,m+e_i+e_j) : m \in \mathbb{Z}^3 , 1 \le i \neq j \le 3 \} .
\]

\noindent
One checks that this collection of squares is complete and associative.
\end{enumerate}
\end{examples}

\section{Computing the fundamental group of a \texorpdfstring{$k$}{k}-graph} \label{sec:pi1}

\noindent In this section we define and provide a  presentation of the fundamental group of a $k$-graph. Kaliszewski, Kumjian, Quigg and Sims show in \cite[Corollary~4.2]{KKQS} that the fundamental group of a $k$-graph may be realised as a quotient of the fundamental group of its skeleton. We provide an alternative proof of this in Theorem~\ref{thm:fgc} which yields a natural presentation of the group. We demonstrate the practical use of our result in Examples~\ref{ex:kps_ex}.

\begin{dfn}[{\cite[Definition~2.8]{KPS3}}]\label{connected}
We say that the $k$-graph $\Lambda$ is \emph{connected} if the equivalence relation on $\Lambda^0$ generated by the relation
$u \sim v$ iff $u \Lambda v \neq \emptyset$ is $\Lambda^0 \times \Lambda^0$.
\end{dfn}

\noindent We review the construction of the fundamental groupoid of a  connected $k$-graph from \cite{PQR1}.
First we describe the fundamental groupoid $\Gg (E)$ of a directed graph $E$.

Following \cite[p. 197]{PQR1}, let $E$ be a directed graph, then a \emph{relation} for $E$ is a pair $( \alpha , \beta )$ of paths in $\mathcal{P} (E)$ such that $s ( \alpha ) = s ( \beta )$ and $r ( \alpha ) = r ( \beta )$. If $K$ is a set of relations for $E$, then $\mathcal{P} (E) / K$ is the quotient of $\mathcal{P} (E)$ by the equivalence relation generated by $K$, for more details see \cite[\S 2]{PQR1}.

As in \cite[Definition 5.2]{PQR1} let $C = \{ ( e^{-1}e,s(e)) : e \in E^u \}$ and call $C$ the set of \textit{cancellation relations} for $E^+$. The quotient $\mathcal{P} (E^+) / C$ is then the fundamental groupoid, $\mathcal{G} (E)$ of $E$. We denote the quotient functor $\mathcal{P} (E^+) \to \mathcal{G} (E)$ by $q_C$. Elements of $\mathcal{G} (E)$ are reduced undirected paths in $E$ with composition given by concatenation followed by cancellation.

Now we turn to defining the fundamental groupoid of a $k$-graph $\Lambda$. First apply the above construction to form $\Gg (E)$ where $E=\Sk_\Lambda$. As in~\cite{HRSW,PQR1}, let $S$ be the equivalence relation on $\Pp(E)$ generated by $\Cc_\Lambda$, the commuting squares in $E$ determined by $\Lambda$. That is, the transitive closure in $\mathcal{P} (E) \times \mathcal{P} (E)$ of
\begin{equation}\label{eq:generating equivalences}
    \begin{split}
    \textstyle \bigcup_{n \ge 2} \{(\mu, \nu) \in E^n \times E^n: {}& \text{there exists } i < n \text{ such that } \\
        &\ \mu_j = \nu_j\text{ whenever } j \not\in \{i, i+1\} \text{ and } \mu_i \mu_{i+1} \sim_{\Cc_{\Lambda}} \nu_i \nu_{i+1}\}.
    \end{split}
\end{equation}

\noindent As in \cite[Observation 5.3]{PQR1} this relation may be extended uniquely to a relation $S^+$ on $\mathcal{P} ( E^+ )$ by adding the relation $( f^{-1} e^{-1} , h^{-1} g^{-1} )$ whenever $(ef,gh) \in S$. This induces a relation, also called $S^+$, on $\mathcal{G} (E)$.

\begin{dfns}
Let $\Lambda$ be a connected $k$-graph. Then the \emph{fundamental groupoid}, $\mathcal{G} ( \Lambda )$ is 
\[
\mathcal{G} ( \Lambda ) := \Gg(\Sk_\Lambda) / S^+ = ( \mathcal{P} ( \Sk_\Lambda^+ ) / C ) / S^+ = \Pp(\Sk_\Lambda^+) / (C \cup S^+).
\]

\noindent For $v \in \Lambda^0$ the \emph{fundamental group based at $v \in \Lambda^0$} is the isotropy group $\pi_1 ( \Lambda , v ) := v \mathcal{G} ( \Lambda ) v$.
\end{dfns}

\noindent
The above definition of the fundamental groupoid of a $k$-graph is consistent  with the one given in \cite[Definition~5.6]{PQR1} (see also the accompanying discussion).

Our goal is to obtain a practical way of giving a presentation of $\pi_1(\Lambda,v)$. First recall that for a connected directed graph $E$, the quotient functor $q_C : \mathcal{P}(E^+ ) \to \Pp(E^+)/ C = \mathcal{G} ( E )$ restricts to $v \mathcal{P} ( E^+ ) v$ and the image is the isotropy group $v\Gg(E)v$, which is by definition the fundamental group of $E$ at $v$, denoted $\pi_1 ( E , v )$.

The following result is well-known (see\cite{St} for instance).

\begin{lem}\label{lem:directed graph fundamental group}
Let $E$ be a connected directed graph, $v \in E^0$, and $T$ be a maximal spanning tree of $E$. 
Then $\pi_1 ( E ,v ) \cong \langle E^1 \mid T^1 \rangle := \langle e \in E^1 \mid e = 1 \text{ if } e \in T^1 \rangle$.
\end{lem}

\begin{proof}
Suppose $e \not\in T^u = T^1 \sqcup T^{-1}$. With notation as in Remark~\ref{rmk: fundamental group ntn}, all the edges of $\eta_{r(e)} , \eta_{s(e)}$ are in $T^u$, so $\zeta_e$ is reduced undirected path in $E$ and hence $q_C ( \zeta_e ) = \zeta_e$. Suppose that $e \in T^u$ then $\zeta_e$ is an undirected path in $T$ from $v$ to $v$ and so its reduced form must be $v$, hence $q_C ( \zeta_e ) = v$.

To complete the proof it suffices to show that $\{ \zeta_e : e \in E^1 \backslash T^1 \}$ freely generate $\pi_1 ( E,v)$. This is a standard result, see \cite[\S2.1.7, \S2.1.8]{St} for example.
\end{proof}

\noindent
Since Lemma~\ref{lem:directed graph fundamental group} holds for any choice of maximal spanning tree, it follows that $\pi_1(E,v)$ does not depend on the choice of basepoint $v$. We denote the fundamental group of a graph $E$ by $\pi_1(E)$. Now we turn our attention to computing the fundamental group of a connected $k$-graph $\Lambda$. Since $\Lambda \cong \Pp ( \Sk_\Lambda ) / S$  we expect the relation $S$ to appear in the description of $\pi_1(\Lambda)$.

\begin{thm} \label{thm:fgc}
Let $\Lambda$ be a connected $k$-graph, $v \in \Lambda^0$ and let $T$ be a maximal spanning tree for $\Sk_\Lambda$. Then $\pi_1 ( \Lambda , v ) \cong \langle \Sk_\Lambda^1 \mid t=1 \text{ if } t \in T^1, ef=gh \text{ if } (ef,gh) \in S^+  \rangle$.
\end{thm}

\begin{proof}
Denote by $q_S :\Gg(\Sk_\Lambda) \to \Gg(\Sk_\Lambda) / S^+ = \Gg(\Lambda)$ the quotient map. With notation as in Remark~\ref{rmk: fundamental group ntn}, observe that $\zeta_e \zeta_f = \zeta_{ef}$ and $\zeta_g \zeta_h = \zeta_{gh}$ in $\mathcal{G} (\Sk_\Lambda)$. Hence $q_S ( \zeta_e \zeta_f ) = q_S ( \zeta_g \zeta_h )$ if and only if $(ef,gh) \in S^+$. Since taking quotients preserves objects, $\pi_1(\Lambda,v) = \pi_1 ( \Sk_\Lambda ) /S^+$. Then Lemma~\ref{lem:directed graph fundamental group} implies that $\pi_1(\Lambda,v) \cong \langle \Sk_\Lambda^1 \mid t=1 \text{ if } t \in T^1, ef=gh \text{ if } (ef,gh) \in S^+  \rangle$.
\end{proof}

\noindent
Since Theorem~\ref{thm:fgc} holds for every choice of $T$, the group $\pi_1(\Lambda,v)$ does not depend on $T$. We henceforth denote by $\pi_1(\Lambda)$ the fundamental group of $\Lambda$. Theorem~\ref{thm:fgc} gives us an explicit presentation of $\pi_1(\Lambda)$, as seen in the following examples.

\vspace{5mm}
\goodbreak

\begin{examples}\label{ex:kps_ex} \hfill

\begin{enumerate}[(i)]

\item 
\begin{minipage}[t]{0.75 \linewidth}
Let $\Sigma$ be the 2-graph, which is completely determined by its $1$-skeleton, shown to the right. In \cite[Example 3.10]{KKQS} it was shown that the topological realisation of $\Sigma$ is the $2$-sphere $S^2$. 
Let $T$ be the maximal spanning tree for $\Sk_\Sigma$ consisting of edges $T^1=\{a,b,c,d,e\}$. 
The commuting squares in $\Sk_\Sigma$ are $(ga,ce)$, $(gb,cf)$, $(de,ha)$ and $(df,hb)$, thus Theorem~\ref{thm:fgc} gives
\[
\pi_1(\Sigma) \cong \langle \Sk_\Sigma^1  \mid  t = 1 \text{ if } t \in T^1, ga=ce,gb=cf,de=ha,df=hb \rangle .
\]

\noindent
The first relation forces $g=1$, the second $g=f$, the third $h=1$ and the fourth $f=h$. Hence all the generators of $\pi_1(\Sigma)$ are equal to $1$. Therefore $\pi_1(\Sigma)$ is trivial.
\end{minipage}%
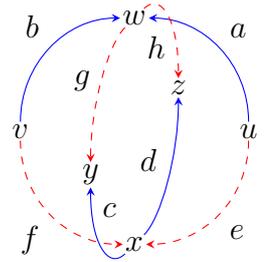
\begin{minipage}[t]{0.25 \linewidth} 
\[
\begin{tikzpicture}[scale=1.5, >=stealth]
    \node[inner sep=0.8pt] (100) at (1,0,0) {$u$};
    \node[inner sep=0.8pt] (-100) at (-1,0,0) {$v$};
    \node[inner sep=0.8pt] (010) at (0,1,0) {$w$};
    \node[inner sep=0.8pt] (0-10) at (0,-1,0) {$x$};
    \node[inner sep=0.8pt] (001) at (0,0,1) {$y$};
    \node[inner sep=0.8pt] (00-1) at (0,0,-1) {$z$};
    \draw[->, blue] (100) .. controls +(0,0.6,0) and +(0.6,0,0) .. (010) node[pos=0.5, anchor=south west] {\color{black}$a$};
    \draw[->, red, dashed] (100) .. controls +(0,-0.6,0) and +(0.6,0,0) .. (0-10) node[pos=0.5, anchor=north west] {\color{black}$e$};
    \draw[->, blue] (-100) .. controls +(0,0.6,0) and +(-0.6,0,0) .. (010) node[pos=0.5, anchor=south east] {\color{black}$b$};
    \draw[->, red, dashed] (-100) .. controls +(0,-0.6,0) and +(-0.6,0,0) .. (0-10) node[pos=0.5, anchor=north east] {\color{black}$f$};
    \draw[->, red, dashed] (010) .. controls +(0,0,0.6) and +(0,0.6,0) .. (001) node[pos=0.5, anchor=east] {\color{black}$g$};
    \draw[->, red, dashed] (010) .. controls +(0,0,-0.6) and +(0,0.6,0) .. (00-1) node[pos=0.8, anchor=east] {\color{black}$h$};
    \draw[->, blue] (0-10) .. controls +(0,0,0.6) and +(0,-0.6,0) .. (001) node[pos=0.85, anchor=west] {\color{black}$c$};
    \draw[->, blue] (0-10) .. controls +(0,0,-0.6) and +(0,-0.6,0) .. (00-1) node[pos=0.5, anchor=south east] {\color{black}$d$};
\end{tikzpicture}
\]
\end{minipage}

\item
\begin{minipage}[t]{0.75 \linewidth}
Consider the 2-graph $\Pi$ with $1$-skeleton $\Sk_\Pi$ shown to the right, with
commuting squares $(ga,ce)$, $(gb,df)$, $(hb,cf)$ and $(ha,de)$. In \cite[Example 3.12]{KKQS} it was shown that the topological realisation of $\Pi$ is the projective plane. Choose spanning tree $T$ of $\Sk_\Pi$ with $T^1=\{a,b,c,f\}$. Then Theorem~\ref{thm:fgc} gives
\[
\pi_1(\Pi)=\langle \Sk_\Pi^1 \mid t=1 \text{ if } t \in T^1, ga=ce, gb=df, hb=cf , ha=de \rangle .
\]

\noindent
The relations become $g=e$, $g=d$, $h=1$ and $de=1$. So the fundamental group of $\pi_1 ( \Pi ) \cong \langle e \mid e^2=1\rangle\cong \mathbb{Z}/2\mathbb{Z}$, the fundamental group of the projective plane.
\end{minipage}%
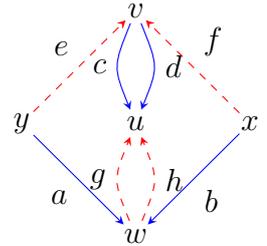
\begin{minipage}[t]{0.25\linewidth} 
\[
\begin{tikzpicture}[scale=1.5, >=stealth]
    \node[inner sep=0.8pt, circle] (u) at (0,0) {$u$};
    \node[inner sep=0.8pt, circle] (v) at (0,1) {$v$};
    \node[inner sep=0.8pt, circle] (w) at (0,-1) {$w$};
    \node[inner sep=0.8pt, circle] (x) at (1,0) {$x$};
    \node[inner sep=0.8pt, circle] (y) at (-1,0) {$y$};
    \draw[->, blue] (v) .. controls +(-0.2,-0.5) .. (u) node[pos=0.5, left, black] {$c$};
    \draw[->, blue] (v) .. controls +(0.2,-0.5) .. (u) node[pos=0.5, right, black] {$d$};
    \draw[->, red, dashed] (w) .. controls +(-0.2,0.5) .. (u) node[pos=0.5, left, black] {$g$};
    \draw[->, red, dashed] (w) .. controls +(0.2,0.5) .. (u) node[pos=0.5, right, black] {$h$};
    \draw[->, blue] (x)--(w) node[pos=0.5, anchor=north west, black] {$b$};
    \draw[->, red, dashed] (x)--(v) node[pos=0.5, anchor=south west, black] {$f$};
    \draw[->, blue] (y)--(w) node[pos=0.5, anchor=north east, black] {$a$};
    \draw[->, red, dashed] (y)--(v) node[pos=0.5, anchor=south east, black] {$e$};
\end{tikzpicture}
\]
\end{minipage}

\item Recall the $1$-skeleton of the $2$-graph $\mathbb{F}^2_\theta$ described in Example~\ref{kgex2}. Since $\mathbb{F}^2_\theta$ has a single vertex, $v$, the maximal spanning tree for its $1$-skeleton $\Sk_{\mathbb{F}^2_\theta}$ is $v$. The commuting squares of $\Sk_{\mathbb{F}^2_\theta}$ are $(f_i g_j , g_{j'} f_{i'} )$ where $\theta (i,j) = (i',j')$ for $(i,j) \in \underline{m} \times \underline{n}$. Hence by Theorem~\ref{thm:fgc} the fundamental group of $\mathbb{F}^2_\theta$ is 
\[
\langle f_i , g_j \mid f_i g_j = g_{j'} f_{i'}  \text{ where } \theta (i,j) = (i',j') \rangle .
\]

\noindent
For different choices of $\theta$ we get quite different fundamental groups:

\begin{enumerate}[(a)]
\item If $m=n=2$ and $\theta : \underline{2} \times \underline{2} \to \underline{2} \times \underline{2}$ is the identity map then 
\[
	\pi_1 ( \mathbb{F}^2_\theta ) \cong \langle f_1 , f_2  \rangle \times \langle g_1 ,g_2  \rangle \cong (\ZZ * \ZZ) \times (\ZZ * \ZZ) 
	= \pi_1 ( S^1 \vee S^1 ) \times \pi_1 ( S^1 \vee S^1 ) ,
\]
where $*$ is the free product, and $\vee$ is the wedge sum.

\item If $m=n=2$ and $\theta$ is given by $\theta(i,j)=(j,i)$  then by Theorem~\ref{thm:fgc} we have
\[
\pi_1 ( \mathbb{F}^2_\theta ) \cong \langle f_1 , f_2 , g_1, g_2  \mid f_1 g_1 = g_1 f_1 , f_1 g_2 = g_1 f_2 , f_2 g_1 = g_2 f_1 , f_2 g_2 = g_2 f_2 \rangle .
\]

\noindent
The first and fourth relations give $f_i^{-1}g_i = g_if_i^{-1}$ for $i=1,2$, then using the second relation we have
$f_1^{-1}g_1=g_2f_2^{-1}$. Putting these together gives $g_1f_1^{-1} = f_2^{-1}g_2$, and hence $f_2g_1 = g_2f_1$. So the third relation is redundant and
\[
	\pi_1(\mathbb{F}^2_\theta) \cong (\ZZ^2 * \ZZ^2)/\langle g_1f_1^{-1} = g_2f_2^{-1} \rangle = \ZZ^2 *_\ZZ \ZZ^2 ,
\]

\noindent
where $\{f_1,g_1\}$ generate the first copy of $\ZZ^2$ and $\{f_2,g_2\}$ generate the second copy, and the amalgamation over $\ZZ$ is with respect to the identifications of $\ZZ$ in $\ZZ^2$ given by $1 \mapsto g_if_i^{-1}$ for $i=1,2$. So $\mathbb{F}^2_\theta$ has the same fundamental group as the two-holed torus.
\end{enumerate}

\item Recall the $3$-graph $\Delta_3$ and its $1$-skeleton described in Examples~\ref{ex:kg}(b) and Examples~\ref{ex:morekg}(b). We claim that
$\pi_1 ( \Delta_3 , (0,0,0) ) $ is isomorphic to the trivial group $\{ 1 \}$. Whilst this follows from the observation that the geometric realisation of $\Delta_3$ is $\mathbb{R}^3$ (see \cite{KKQS}), it is instructive to do the computation using Theorem~\ref{thm:fgc}, without reference to this observation. We choose spanning tree for $\Sk_\Lambda$ consisting of the lines $x=y=0$ ($z$-axis, green), $x=0,z=k$ for $k \in \mathbb{Z}$ ($y$-axis at each $z$-level, red) and $y=k$ for $k \in \mathbb{Z}$  (constant $y$-lines on each $z$-level, blue).

First consider the row of cubes with bases $S_{0j0}$, $j \in \mathbb{Z}$ along the $y$-axis as shown below 
\[
\begin{tikzpicture}[->=stealth]
\coordinate (O) at (0,0,0);
\coordinate (A) at (0,\Width,0);
\coordinate (B) at (0,\Width,\Height);
\coordinate (C) at (0,0,\Height);
\coordinate (D) at (\Depth,0,0);
\coordinate (E) at (\Depth,\Width,0);
\coordinate (F) at (\Depth,\Width,\Height);
\coordinate (G) at (\Depth,0,\Height);

\draw[blue] (C) -- (O);
\draw[red,dashed] (G) -- (C);
\draw[blue] (G) -- (D);
\draw[red] (D) -- (O); 

\draw[blue] (B) -- (A);
\draw[red,dashed] (F) -- (B);
\draw[blue] (F) -- (E);
\draw[red] (E) -- (A);

\draw[green!50!black,dotted]  (E) -- (D);
\draw[green!50!black,dotted]  (B) -- (C);
\draw[green!50!black,dotted]   (F) -- (G);
\draw[green!50!black] (A) -- (O);

\node at (2.2,1,0) {\Tiny $z_{011}$};
\node at (-2.3,1,0.2) {\Tiny $z_{0-11}$}; %
\node at (0,1.3,2.6) {\Tiny $z_{101}$}; 
\node at (-2.2,1.3,2.5) {\Tiny $z_{1-11}$};
\node at (-1,0,2.5) {\Tiny $y_{100}$}; 
\node at (1.5,0,2.5) {\Tiny $y_{110}$}; 
\node at (0.7,1.4,0) {\Tiny $y_{111}$}; 
\node at (-1.4,1.4,0) {\Tiny $y_{101}$};
\node at (1,1,2.2) {\Tiny $(0,0,0)$};
\node at (1,0.7,0) {\Tiny $z_{111}$};

\node at (2.5,0,0) {\Tiny $(0,1,0)$};
\node at (0,2.2,0) {\Tiny $(0,0,1)$};
\node at (0,0,2.5) {\Tiny $(1,0,0)$};

\node at (1,0,1) {\Tiny $S_{010}$};
\node at (-1,0,1) {\Tiny $S_{000}$};

\node at (3.2,0,0) {\tiny $\ldots$};
\node at (2.4,2,0) {\tiny $\ldots$};
\node at (-2.2,0,0) {\tiny $\ldots$};
\node at (-2.4,2,0) {\tiny $\ldots$};

\node at (2.4,2,2) {\tiny $\ldots$};
\node at (2.4,0,2) {\tiny $\ldots$};
\node at (-2.2,2,2) {\tiny $\ldots$};
\node at (-2.4,0,2) {\tiny $\ldots$};

\draw[red] (A) -- (-2,2,0);
\draw[red] (O) -- (-2,0,0);
\draw[red,dashed] (B) -- (-2,2,2);
\draw[red,dashed] (C) -- (-2,0,2);

\draw[green!50!black,dotted] (-2,2,0) -- (-2,0,0);
\draw[green!50!black,dotted] (-2,2,2) -- (-2,0,2);

\draw[blue] (-2,2,2) -- (-2,2,0);
\draw[blue] (-2,0,2) -- (-2,0,0);

\end{tikzpicture}
\]

\noindent
The solid lines shown above are part of the spanning tree. For the cube with base $S_{010}$ we have $y_{110} = y_{111}= 1$ as three sides are trivial on the top and bottom faces; then $z_{101} = z_{011} = 1$ as three sides are now trivial of the left and back faces, finally $z_{111} =1$ as three sides are now trivial of the right face. Considering the cube with base $S_{000}$ we find $y_{100} = y_{101} =1$ and $z_{1-11} = z_{0-11} =1$ by the same reasoning. Now we may repeat this for the cubes moving in the positive and negative directions on the $y$-axis.

Now, moving in the positive $x$-direction and considering the line of cubes with bases $S_{1j0}$, $j \in \mathbb{Z}$, we get a similar picture except the green lines on the back face are all known to be trivial, and hence all edges on this line of cubes $S_{1j0}$, $j\in \mathbb{Z}$ are trivial. Continuing in this way we find that the edges on the line of cubes $S_{ij0}$, $i \in \mathbb{N}, j \in \mathbb{Z}$ are trivial.

Now, moving the line of cubes $S_{0j0}$, $j \in \mathbb{Z}$ in the negative $x$-direction,
 and considering the line of cubes with bases $S_{-1j0}$, $j \in \mathbb{Z}$, we get a similar picture except the green lines on the front face are all known to be trivial, and hence all edges on this line of cubes are trivial. Continuing in this way we find that the edges on the line of cubes $S_{ij0}$, $i , j \in \mathbb{Z}$ are trivial.

Now, moving the line of cubes $S_{0j0}$, $j \in \mathbb{Z}$ in the positive  $z$-direction and considering the line of cubes with bases $S_{0j1}$, $j \in \mathbb{Z}$, we get a similar picture except the red lines on the bottom face are all known to be trivial, and hence all edges on this line of cubes are trivial.  Continuing in this way we find that the edges on the line of cubes $S_{0jk}$, $k \in \mathbb{N}, j \in \mathbb{Z}$ are trivial.

Now, moving the line of cubes $S_{0j0}$, $j \in \mathbb{Z}$ in the negative
$z$-direction and considering the line of cubes with bases $S_{0j-1}$, $j \in \mathbb{Z}$, we get a similar picture except the red lines on the top face are all known to be trivial, and hence all edges on this line of cubes are trivial. Continuing in this way we find that the edges on the line of cubes $S_{0jk}$, $j, k \in \mathbb{Z}$ are trivial.

Since every cube with base $S_{ijk}$ is in some line of cubes of some $x$- or $z$-translate of our initial line of cubes it follows that all edges in the skeleton are trivial in the fundamental group. This estalishes our claim.
\end{enumerate}
\end{examples}

\section{Relationship with first Homology group} \label{sec:h1}

Full versions of the following defintions may be found in \cite[\S 3]{KPS3}.
Let $X$ be a set. We write $\mathbb{Z} X$ for the free abelian group generated by $X$. For a $k$-graph $\Lambda$, set $C_0 ( \Lambda ) = \mathbb{Z} \Lambda^0$, $C_1 ( \Lambda ) = \mathbb{Z} \Lambda^{e_1}  \oplus \cdots \oplus \mathbb{Z} \Lambda^{e_k}$ and $C_2 ( \Lambda ) = \oplus_{1\le i<j\le k} \mathbb{Z} \Lambda^{e_i+e_j}$. 

Let $\partial^\Lambda_1 : C_1 ( \Lambda ) \to C_0 ( \Lambda )$ be the  homomorphism determined by $\partial^\Lambda_1 (\lambda) = s(\lambda)-r(\lambda)$.  Define $\partial^\Lambda_2 : C_2 ( \Lambda ) \to C_1 ( \Lambda )$ as follows. Suppose $\lambda \in \Lambda^{e_i+e_j}$ where $1\le i<j\le k$. Factorise $\lambda =f_1 g_1=g_2 f_2$ where $f_r \in  \Lambda^{e_{i}}$ and $g_r \in \Lambda^{e_{j}}$ for $r=1,2$, then set 
$\partial_2^\Lambda (\lambda)=f_1+g_1-f_2-g_2$
and extend to a homomorphism from $C_2 ( \Lambda )$ to $C_1 ( \Lambda )$. Then $\partial_2^\Lambda \circ \partial_1^\Lambda =0$, and $H_0 (\Lambda) =\mathbb{Z} \Lambda^0 / \operatorname{Im} \partial^\Lambda_1$, $H_1 (\Lambda) = \operatorname{ker} \partial^\Lambda_1 /\operatorname{Im} \partial^\Lambda_2$.

Recall the following definitions from \cite[Definition 2.7, Definition 3.10]{KPS3}.

\begin{dfn}\label{rmk:trails}
Given $h = h_1^{m_1} \cdots h_n^{m_n} \in \Gg ( \Lambda )$, where $m_i = \pm 1$, define $t : \Gg ( \Lambda ) \to C_1 ( \Lambda )$ by  $t(h) = \sum_{i=1}^n m_i h_i \in C_1 ( \Lambda )$; then $t(h)$ is called a \emph{trail}. If $h$ is a circuit (that is $r(h)=s(h)$) then $t(h)$ is called a \emph{closed trail}. If $h$ is also simple (that is $s ( h_i^{m_i} ) \neq s ( h_j^{m_j} )$ for $i \neq j$),then $t(h)$ is called a \emph{simple closed trail}.
\end{dfn}

\goodbreak

\begin{prop} \label{prp:h1pi1}
Let $\Lambda$ be a connected $k$-graph. Then the map $t$ defined in Definition~\ref{rmk:trails} induces an isomorphism $\operatorname{Ab} \pi_1 ( \Lambda ) \cong H_1 ( \Lambda )$.
\end{prop}

\begin{proof}
Fix $v \in \Lambda^0$, then $\pi_1(\Lambda) \cong \pi_1(\Sk_\Lambda,v)/ S^+$ by~\cite{KKQS}. Fix a maximal spanning tree $T \subset \Sk_\Lambda$, then $\pi_1(\Sk_\Lambda,v) = \langle \zeta_e \mid e \in E^1 \setminus T^1 \rangle$ (see \cite{St} for example). Then $t: \pi_1(\Sk_\Lambda,v) \to C_1(\Lambda)$ is a homomorphism. Since $t$ sends simple reduced circuits to simple closed trails,  \cite[Proposition~3.15]{KPS3} implies that $\ker(\partial_1^\Lambda) = t(\pi_1(\Sk_\Lambda,v))$. As $t ( \zeta_e \zeta_f ) - t ( \zeta_g \zeta_h ) = e+f-g-h \in \operatorname{Im} \partial_2^\Lambda$ whenever if $(ef,gh) \in S^+$, $t$ descends to a homomorphism $t':\pi_1(\Lambda) \to H_1(\Lambda)$ which maps $[a]$ to $[t(a)]$ for $a \in \pi_1 ( \Sk_\Lambda ,v)$. Routine calculation then shows that $\ker t'$ is the commutator subgroup of $\pi_1(\Lambda)$,  so 
$t$ is an isomorphism from $\operatorname{Ab}\pi_1(\Lambda)$, the abelianisation of $\pi_1 ( \Lambda)$ to $H_1(\Lambda)$.
\end{proof}

\begin{example} \label{ex:isklein}
Recall the $2$-graph $\Lambda$ from \cite[Example 5.7]{KPS3}
\[
\begin{tikzpicture}[scale=1.5]
    \node[inner sep=0.5pt, circle] (sw) at (-1,-1) {$x$};
    \node[inner sep=0.5pt, circle] (w) at (-1,0) {$v$};
    \node[inner sep=0.5pt, circle] (nw) at (-1,1) {$x$};
    \node[inner sep=0.5pt, circle] (s) at (0,-1) {$w$};
    \node[inner sep=0.5pt, circle] (m) at (0,0) {$u$};
    \node[inner sep=0.5pt, circle] (n) at (0,1) {$w$};
    \node[inner sep=0.5pt, circle] (se) at (1,-1) {$x$};
    \node[inner sep=0.5pt, circle] (e) at (1,0) {$v$};
    \node[inner sep=0.5pt, circle] (ne) at (1,1) {$x$};
    \draw[-latex, blue] (sw)--(s) node[pos=0.5, above, black] {$b$};
    \draw[-latex, blue] (se)--(s) node[pos=0.5, above, black] {$a$};
    \draw[-latex, blue] (w)--(m) node[pos=0.5, above, black] {$c$};
    \draw[-latex, blue] (e)--(m) node[pos=0.5, above, black] {$d$};
    \draw[-latex, blue] (nw)--(n) node[pos=0.5, above, black] {$b$};
    \draw[-latex, blue] (ne)--(n) node[pos=0.5, above, black] {$a$};
    \draw[-latex, red, dashed] (sw)--(w) node[pos=0.5, right, black] {$f$};
    \draw[-latex, red, dashed] (nw)--(w) node[pos=0.5, right, black] {$e$};
    \draw[-latex, red, dashed] (s)--(m) node[pos=0.5, right, black] {$h$};
    \draw[-latex, red, dashed] (n)--(m) node[pos=0.5, right, black] {$g$};
    \draw[-latex, red, dashed] (se)--(e) node[pos=0.5, right, black] {$e$};
    \draw[-latex, red, dashed] (ne)--(e) node[pos=0.5, right, black] {$f$};
    \node at (-0.5, -0.5) {$\gamma$};
    \node at (-0.5, 0.5) {$\alpha$};
    \node at (0.5, -0.5) {$\beta$};
    \node at (0.5, 0.5) {$\delta$};
\begin{scope}[xshift=3cm, yshift=-0.78cm, scale=1.5]
    \node[inner sep=0.5pt, circle] (u) at (0,0) {$u$};
    \node[inner sep=0.5pt, circle] (v) at (1,0) {$v$};
    \node[inner sep=0.5pt, circle] (w) at (0,1) {$w$};
    \node[inner sep=0.5pt, circle] (x) at (1,1) {$x$};
    \draw[blue,-latex,out=160, in=20] (v) to node[pos=0.5,above, black] {$c$} (u);
    \draw[blue,-latex,out=200, in=340] (v) to node[pos=0.5,below, black] {$d$} (u);
    \draw[blue,-latex,out=160, in=20] (x) to node[pos=0.5,above, black] {$a$} (w);
    \draw[blue,-latex,out=200, in=340] (x) to node[pos=0.5,below, black] {$b$} (w);
    \draw[red, dashed, -latex, out=290, in=70] (w) to node[pos=0.5,right, black] {$h$} (u);
    \draw[red, dashed, -latex, out=250, in=110] (w) to node[pos=0.5,left, black] {$g$} (u);
    \draw[red, dashed, -latex, out=290, in=70] (x) to node[pos=0.5,right, black] {$f$} (v);
    \draw[red, dashed, -latex, out=250, in=110] (x) to node[pos=0.5,left, black] {$e$} (v);
\end{scope}
\end{tikzpicture}
\]
which has the same homology as the Klein bottle. However, as we shall see, it does have the same fundamental group, but with a quite different presentation to the one given in \cite[Example 3.13]{KKQS}. To see this chose spanning tree $T$ with $T^1=\{a,c,g\}$. By Theorem~\ref{thm:fgc} the fundamental group is generated by $\Lambda^{e_1} \cup \Lambda^{e_2}$ subject to the relations
\[
a=c=g=1 , \ gb=ce , \ ga = df , \ hb = cf , \ ha = de ,
\]

\noindent which simplify to
\[
b=e , \ 1 = df , \ hb = f , \ h = de .
\]

\noindent
Eliminating $b$ and simplifying further we have
\begin{equation} \label{eq:notklein}
\pi_1 ( \Lambda ) = \langle  e, f , h : fh = e , he =f \rangle  
= \langle e, f : f^2 = e^2 \rangle   ,
\end{equation}

\noindent
is equal to the fundamental group of the Klein bottle, $\langle a,b : aba=b \rangle$. To see this set $e=ab$ and $f=b$, then
\[
e^2 = (ab)(ab) = (aba)b = (b)(b) = b^2 = f^2 .
\]

\noindent
A slightly easier calculation shows that in the case $n=2$, the $2$-graph in \cite[Example 5.1]{KPS3} has the same fundamental group \eqref{eq:notklein} as $\Lambda$, which is not a suprise as it has the same topological realisation as $\Lambda$ (see \cite[Remark 5.9]{KPS3}). The presentation \eqref{eq:notklein} in abelian form is
\[
 \langle e , f : 2 (f-e) =0 \rangle . 
\]

\noindent
One sees that the abelianisation of $\pi_1 ( \Lambda )$ is $\mathbb{Z} \oplus \mathbb{Z} / 2 \mathbb{Z}$, the homology group of the Klein bottle, as stated in \cite[Example 5.7]{KPS3}.
\end{example}

\end{document}